\documentclass[12pt, reqno]{amsart}
\usepackage{amsmath, amsthm, amscd, amsfonts, amssymb, graphicx, color}
\usepackage[bookmarksnumbered,plainpages]{hyperref}
\newtheorem{theorem}{Theorem}[section]

\newtheorem{proposition}[theorem]{Proposition}
\newtheorem{corollary}[theorem]{Corollary}
\theoremstyle{definition}
\newtheorem{definition}[theorem]{Definition}

\newtheorem{examples}[theorem]{Examples}
\theoremstyle{remark}

\newcommand{\A}{\mathcal A}

\newcommand{\bea}{\begin{eqnarray*}}
\newcommand{\eea}{\end{eqnarray*}}
\begin{document}
\title {  CHARACTER INNER AMENABILITY OF CERTAIN BANACH ALGEBRAS}
\author{  H.R. EBRAHIMI VISHKI and A.R. KHODDAMI  }
\subjclass[2000]{43A60, 43A07} \keywords{ $\varphi$-amenability,
 $\varphi$-inner amenability,  character inner amenability, tensor product, Lau algebra, triangular Banach algebra,  module extension.}
\address{Department of Pure Mathematics And Centre Of Excellence In Analysis On Algebraic Structures (CEAAS), Ferdowsi  University Of Mashhad, P.O. Box 1159, Mashhad 91775, Iran}
 \email{Vishki@um.ac.ir}
\address{Department of Pure Mathematics, Ferdowsi University of Mashhad, P.O. Box 1159, Mashhad 91775, Iran} \email{khoddami.alireza@yahoo.com}
\begin{abstract}Character inner amenability for  certain class of Banach algebras consist of projective tensor product $A\hat{\otimes}B$, Lau product $A\times_\theta B$ and  module extension $A\oplus X$ are investigated. Some illuminating examples are also included.      \end{abstract}
\maketitle
\section{INTRODUCTION }
The concept of left amenability for a Lau algebra (a predual of a von Neumann  algebra for which the identity of the dual is a multiplicative linear functional, \cite{L})
has been extensively extended  for an arbitrary Banach algebra by introducing the notion of $\varphi-$amenability in  Kaniuth {\em et al.} \cite{KLP}.  A  Banach algebra  $A$ was  called $\varphi$-amenable ($\varphi\in\triangle(A)=$  the spectrum of  $A$)
if there exists a $m\in A^{**}$  satisfying $m(\varphi)=1$ and
$m(f\cdot a)=\varphi(a)m(f)$  $(a\in A,\ f\in A^*)$. $A$ was called character amenable if it is $\varphi$-amenable for each $\varphi\in\triangle(A)$. Many aspects of $\varphi$-amenability have been investigated in \cite{KLP2, M2, HMT}. Recently Jabbari {\it et al.} \cite{JMZ} have introduced the $\varphi-$version of inner amenability.   A Banach algebra $A$ was
  said to be $\varphi$-inner amenable if there exists a $m\in A^{**}$ satisfying
   $m(\varphi)=1$  and   $m(f\cdot a)=m(a\cdot f)$   $(f\in A^*,
  a\in A)$. Such a $m$ will sometimes be referred to as a $\varphi-$inner mean and  $A$ is said to be character inner amenable if and only if $A$ is $\varphi-$inner amenable for every  $\varphi\in\triangle(A).$  As they have  remarked in \cite[Remark 2.4]{JMZ}, this concept considerably generalizes the notion of inner amenability for Lau algebras which  introduced by Nasr-Isfahani \cite{I}. They also  gave several characterizations of $\varphi$-inner amenability. For instance, as in the case of $\varphi-$amenability in  \cite[Theorem 1.4]{KLP}, they have shown that a $\varphi-$inner mean is in fact some  $w^*$-cluster point of a bounded net  $(a_\alpha)$ in $A$ satisfying $\|a_\alpha a-aa_\alpha\|\rightarrow 0,$ for all $a\in A$ and $\varphi(a_\alpha)=1$ for all $\alpha$; \cite[Theorem 2.1]{JMZ}.

  In this paper, we are going to investigate the character inner amenability for certain products of Banach algebras consist of projective tensor product $A\hat{\otimes}B$, Lau product $A\times_\theta B$ and the  module extension $A\oplus X.$ For instance, we show that the projective tensor product   $A\hat{\otimes}B$  is character inner amenable if and only if both $A$ and $B$ enjoy the same property.

\section{Preliminary results and Examples }
Before we proceed for the results we  need a bit preliminaries. The second dual $A^{**}$ of a Banach algebra $A$ can be made into Banach algebra under each of Arens products $\square$ and $\lozenge$ which are defined as follows.
  For $a ,b \in A, \ f\in A^*$  and  $\ m,n \in
  A^{**}$,   $$\langle m\square n,f \rangle=\langle m,n\cdot f \rangle,  \ \langle n\cdot f,a \rangle=\langle n,f\cdot a \rangle,  \ \langle f\cdot a,b \rangle=\langle f,ab \rangle;$$
  and   $$ \langle f,m\lozenge n \rangle=\langle f\cdot m,n \rangle,   \ \langle a,f\cdot m \rangle=\langle a\cdot f,m \rangle,
    \ \langle b,a\cdot f \rangle=\langle ba,f \rangle.$$
  \vskip 0.4 true cm
We commence with  the next definition from \cite{JMZ}.
\begin{definition}
Let $A$ be a Banach algebra and let $\varphi\in\triangle(A)$.  Then $A$ is called
$\varphi$-inner amenable  if there exists a $m\in A^{**}$
 such that   $m(\varphi)=1$    and
 $m\square a=a\square m,  \ (a\in A)$. We call such a $m$  a  $\varphi$-inner mean.
  A Banach algebra $A$ is called character inner amenable if $A$ is $\varphi$-inner amenable for all $\varphi\in \bigtriangleup(A)$.
\end{definition}
 The next straightforward characterization of $\varphi-$inner amenability (see {\cite[Theorem 2.1]{JMZ}} which is inspirited from {\cite[Theorem 1.4 ]{KLP}})  will be frequency used  in the sequel.
\begin{proposition}\label{K}
Let $A$ be a Banach algebra and  $\varphi\in \bigtriangleup(A)$.
Then
the following statement are equivalent. \\
(i)   $A$ is   $\varphi$-inner amenable.\\
(ii)  There exists a bounded net $(a_\alpha)$ in $A$ such that
$\|aa_\alpha-a_\alpha a\|\rightarrow 0$  for all $a\in A$ and
$\varphi(a_\alpha)=1$  for all $\alpha$.\\
(iii) There exists a bounded net $(a_\alpha)$ in $A$ such that
$\|aa_\alpha-a_\alpha a\|\rightarrow 0$  for all $a\in A$ and
$\varphi(a_\alpha)\rightarrow 1$.
\end{proposition}

\begin{examples}\label{H}
$(i)$  Every Banach algebra with a bounded  approximate identity
$(e_\alpha)$  is character inner amenable. Indeed,  one can verify that $\|ae_\alpha-e_\alpha a\|\rightarrow 0$ and $\varphi(e_\alpha)\rightarrow 1$, for each $\varphi\in\triangle(A)$.

$(ii)$  Every commutative Banach algebra  is character inner amenable.

$(iii)$   Let $A=\left\{  \left(
                        \begin{array}{cc}
                          0 & a \\
                          0 & b \\
                        \end{array}
                      \right):     a, b\in \mathbb{C}    \right
                      \}$ and define  $\varphi :A\rightarrow\mathbb{C}$
                      by $\varphi(\left(
                                            \begin{array}{cc}
                                              0 & a \\
                                              0 & b \\
                                            \end{array}
                                          \right))=b$. A direct verification reveals that there is no
                                          bounded net $(a_\alpha)$ in $A$ satisfying Proposition \ref{K}.
                                          Therefore $A$ is not
                                          $\varphi$-inner
                                          amenable.

$(iv)$ Given a Banach space $A$ and fix a non-zero $\varphi\in A^*$ with $\|\varphi\|\leq 1$. Then the product $a\cdot b=\varphi(a)b$ turning $A$ into a Banach algebra with $\triangle(A)=\{\varphi\}$. Trivially $A$ has a left identity (indeed, every $e\in A$ with $\varphi(e)=1$ is a left identity), while it has no bounded approximate identity in the case where $\dim(A)>1.$ In this case $A$ is not $\varphi$-inner amenable. Indeed, if  $m$ is  a $\varphi$-inner mean for $A$ then  $ m(\varphi)=1$ and $m\square a=a\square m$ for all $a\in A$. But a simple calculation reveals that $m\square a=m(\varphi)a$ and  $a\square m=\varphi(a)m.$  Therefore $a=\varphi(a)m$ for each $a\in A$, that is, $\dim(A)=1.$

$(v)$ Let  $A$ be the Banach algebra posed in $(iv)$ which is  generated by two elements $a$ and $b$, that is, $\dim(A)=2$ and let $\varphi\in A^*$ be so that $\varphi(a)=1$ and $\varphi(b)=0$. If $I$ is the subspace generated by $b$ then  $I$ is a closed ideal  for which $I$ and $A/I$ are character inner amenable; however, $A$ itself is not.

 Note that, as {\cite[Theorem 2.8]{JMZ}} demonstrates, if $A$ is character inner amenable then so is $A/I$ for each closed ideal $I$ of $A$. However, $I$ may not be character inner amenable; for example, the unitization of a non-character inner amenable Banach algebra is character inner amenable.

$(vi)$ For a reflexive Banach algebra $A$ with  $\varphi\in\triangle(A)$ it is easy to verify that $A$ is $\varphi$-inner amenable  if and only if $Z(A)\cap
(A-ker\varphi)\neq\emptyset$, where $Z(A)$ is the algebraic center of $A$.
\end{examples}

\section{Projective  tensor product $A\hat{\otimes}B$}
Let $A\hat{\otimes}B$ be the projective tensor product of two Banach
algebras $A$ and $B$.  For    $f\in A^*$  and   $ g\in B^* $,  let
$f\otimes g $ denote the element of
  $(A\hat{\otimes }B)^*$  satisfying,  $(f\otimes
  g)(a\otimes b)=f(a)g(b)$   $(a\in A,  b\in B)$. Recall
  that,

  $$\bigtriangleup (A\hat{\otimes} B)=\{  \varphi\otimes\psi,    \varphi\in \triangle(A),   \psi\in \triangle(B)  \}.$$
  In the next result, as in the case of character amenability in  {\cite[Theorem 3.3]{KLP}, we investigate  the character inner amenability of $A\hat{\otimes}B.$ It is worthwhile mentioning that our method of proof provides an alternative proof for {\cite[Theorem 3.3]{KLP} which does not rely on derivation techniques.
\begin{theorem}
Let $A$ and $B$ be Banach algebras and let $\varphi\in
\triangle(A)$, $\psi\in \triangle(B)$. Then
  $A\hat{\otimes}B$ is $(\varphi\otimes\psi)$-inner amenable if and
only if $A$  is $\varphi$-inner amenable and  $B$  is   $\psi$-inner
amenable. In particular,    $A\hat{\otimes}B$ is character inner
amenable if and only if both  $A$ and $B$ are character inner
amenable.
\end{theorem}
\begin{proof}
  Let $m\in (A\hat{\otimes }B)^{**}$  be a
$\varphi\otimes\psi$-inner mean. So $m(\varphi\otimes\psi)=1$ and
$$m((f\otimes\psi)\cdot (a\otimes b))=m((a\otimes b)\cdot (f\otimes\psi)), \
(f\in A^*,  a\in A, \  b\in B).$$  Define $m_\varphi:A^*\rightarrow
 \mathbb{C}$  such that  $m_\varphi(f)=m(f\otimes\psi)$.  So
$ m_\varphi(\varphi)=m(\varphi\otimes\psi)=1$. Choose $ b_0\in B$
such that $\psi(b_0)=1$  and  let $f\in A^*$  and   $a\in A$. So
\begin{eqnarray*}
m_\varphi(f\cdot a)&=&m(f\cdot a\otimes\psi)=m(f\cdot a\otimes\psi\cdot b_0)\\
&=& m((f\otimes\psi)\cdot (a\otimes b_0))=m((a\otimes
b_0)\cdot (f\otimes\psi))\\
&=&m(a\cdot f\otimes b_0\cdot \psi)=m(a\cdot f\otimes\psi)\\
 &=&m_\varphi(a\cdot f).
\end{eqnarray*}
It follows that $A$ is
$\varphi$-inner amenable, and similarly $B$ is $\psi$-inner amenable.

For the converse let $A$ be $\varphi$-inner amenable and $B$
$\psi$-inner amenable. Then  there exist bounded nets
$(a_\alpha)$  in $A$  and $(b_\beta)$  in $B$ such that
$\varphi(a_\alpha)=1$, $\|aa_\alpha -a_\alpha a\|\rightarrow 0 $,
$(a\in A)$  and $\psi(b_\beta)=1$,  $\| bb_\beta-b_\beta b
\|\rightarrow 0$, $(b\in B)$.  Consider the bounded net
$(a_\alpha\otimes b_\beta)$ in $A\hat{\otimes}
B$.  So $(\varphi\otimes\psi)(a_\alpha\otimes
b_\beta)=\varphi(a_\alpha)\psi(b_\beta)=1$.  Now let $\| a_\alpha\|$
$\leq M_1$, $\| b_\beta \| \leq M_2 $ and let
$F=\displaystyle{\sum_{j=1}^{N}  c_j\otimes d_j} \in A\otimes
B$.
\begin{eqnarray*}
\| F(a_\alpha\otimes b_\beta)- (a_\alpha\otimes b_\beta )F
\|&=&\|   \displaystyle{\sum_{j=1}^{N}[(c_ja_\alpha -a_\alpha
c_j)\otimes d_jb_\beta +a_\alpha c_j\otimes(d_jb_\beta -b_\beta
d_j)] }  \| \\
 &\leq& \displaystyle{\sum_{j=1}^{N} M_2\|d_j\|\|
c_ja_\alpha-a_\alpha c_j\|}+\displaystyle{\sum_{j=1}^{N} M_1\|
c_j\|\|d_jb_\beta-b_\beta d_j\|}.
\end{eqnarray*}
 Since
$\| c_ja_\alpha-a_\alpha c_j\|\rightarrow 0$ and  $\|d_jb_\beta-b_\beta d_j\| \rightarrow 0, (1\leq j\leq N),$\\
so   $\| F(a_\alpha\otimes
b_\beta)-(a_\alpha\otimes b_\beta)F\|\rightarrow 0$.

 Now
 let $w\in A\hat{\otimes}B$,  so there exist
sequences $ \{c_j\}\subseteq A $ and $ \{d_j\}\subseteq B $  such
that $w=\displaystyle{\sum_{j=1}^{\infty}  c_j\otimes}d_j $ with
$\displaystyle{\sum_{j=1}^{\infty}  \| c_j\|\| d_j\|}<\infty$.
Let $ \epsilon> 0$  be given,  we choose  $N\in \mathbb{N}$ such that
$\displaystyle{\sum_{j=N+1}^{\infty}\|
c_j\|\|d_j\|}<\epsilon/4M_1M_2$. Put
$F=\displaystyle{\sum_{j=1}^{N} c_j\otimes d_j}$. As
$\|F(a_\alpha\otimes b_\beta)-(a_\alpha\otimes
b_\beta)F\|\rightarrow 0 $, so there exists  $(\alpha_0,\beta_0)$
such that   $\|F(a_\alpha\otimes b_\beta)-(a_\alpha\otimes
b_\beta)F\|<\epsilon/2$ for all $(\alpha,\beta)\geq
(\alpha_0,\beta_0)$.\\ Now for such a $(\alpha,\beta)$,
\begin{eqnarray*}
\|w(a_\alpha\otimes b_\beta)-(a_\alpha\otimes
b_\beta)w\|&=&\|F(a_\alpha\otimes b_\beta)-(a_\alpha\otimes
b_\beta)F+\displaystyle{\sum_{j=N+1}^{\infty}[c_ja_\alpha\otimes
d_jb_\beta -a_\alpha c_j\otimes b_\beta
d_j]}\|\\
&\leq&\|F(a_\alpha\otimes b_\beta)-(a_\alpha\otimes
b_\beta)F\|+2M_1M_2\displaystyle{\sum_{j=N+1}^{\infty}\|c_j\|\|d_j\|}\\
&<&\epsilon/2+2M_1M_2.\epsilon/4M_1M_2\\
&=&\epsilon.
\end{eqnarray*}
 Hence $\|w(a_\alpha\otimes
 b_\beta)-(a_\alpha\otimes b_\beta)w\|\rightarrow 0$. Applying Proposition \ref{K} shows  that
$A\hat{\otimes}B$ is $(\varphi\otimes\psi)$-inner amenable.
\end{proof}
\section{The Lau product $A\times_\theta B$}
Let $A$ and $B$ be two Banach algebras  with $\triangle
(B)\neq\emptyset.$
  For a    $\theta\in \triangle (B)$ the $\theta$-Lau product
  $A\times_\theta B$  is  defined as the  cartesian product  $A\times
  B$  with the algebra multiplication  \\
  $(a,b)\cdot (c,d)=(ac+\theta(d)a+\theta(b)c,bd)$ and with the norm  $\|(a,b)\|=\|a\|+\|b\|$. \\  This  product was first  introduced by Lau \cite{L} for Lau algebras and followed by Sangani Monfared \cite{M} for the general case.
 $A\times _\theta B$  is a  Banach algebra and it is shown in {\cite[Proposition 2.4]{M}} that $$ \triangle (A\times_\theta
 B)=(\bigtriangleup(A)\times\{\theta\})\cup(\{0\}\times\bigtriangleup(B)).$$
 In a natural way the dual space  $(A\times_\theta B)^* $ can be
 identified with $A^*\times B^*$ via
 $(f,g)((a,b))=f(a)+g(b)$.  Recall that  the dual norm on    $A^*\times
 B^*$  is  $\|(f,g)\|=max \{  \|f\|  , \|g\|   \}$.
 Also  if  $A^{**}  ,B^{**}$   and   $(A\times_\theta B)^{**}$  are
 equipped with their first    Arens   products then    $(A\times_\theta B)^{**}=A^{**}\times_\theta
 B^{**}$ as an isometric isomorphism.  Also for  $(m,n),(p,q)\in (A\times_\theta B)^{**}$  we have
$(m,n)\square(p,q)=(m\square p+n(\theta)p+q(\theta)m,n\square
q)$; see {\cite[Proposition 2.12]{M}}.

The next result, which  extends  {\cite[Proposition 4.2]{I}}, studies character inner amenability of $A\times_\theta B$.
\begin{theorem}
Let $\varphi\in \bigtriangleup(A)$ and $\psi\in \bigtriangleup(B)$.
Then \\
(i) $A\times_\theta B$ is $(\varphi,\theta)$-inner amenable if and
only if either $A$ is $\varphi$-inner amenable or $B$ is $\theta$-inner
amenable.\\
(ii) $A\times_\theta B$  is  $(0,\psi)$-inner amenable
 if and only if  $B$  is   $\psi$-inner amenable.\\
(iii)  $A\times_\theta B$ is character inner amenable if and only if
$B$ is character inner amenable.\\
\end{theorem}
\begin{proof}
(i) Let $A\times_\theta B$ be $(\varphi,\theta)$-inner amenable. Then
there exists $(m,n)\in \A^{**}\times_\theta B^{**}$ such that
$(m,n)((\varphi,\theta))=1$ and
$(m,n)\square(a,b)=(a,b)\square(m,n),$ for all $(a,b)\in
A\times_\theta B$. It follows that $m(\varphi)+n(\theta)=1,$
$m\square a=a\square m$ and  $n\square b=b\square n$  for all
$a\in A $ and $b\in B$. Now if $n(\theta)=0 $ then $m(\varphi)=1$
and so $m$ is a $\varphi$-inner mean for $A$. If $n(\theta)\neq 0$  then
$\frac{n}{n(\theta)}\square b=b\square \frac{n}{n(\theta)}$, that is,
$\frac{n}{n(\theta)}$ is a $\theta$-inner mean for $B.$ \\
For the converse, suppose that $m$ is a $\varphi$-inner mean for  $A$ then trivially $(m,0)$ is a $(\varphi,\theta)$-inner mean  for $A\times_\theta B.$ The same argument needs for the case that $B$ is a $\psi-$inner amenable. $(ii)$ needs a similar proof and $(iii)$ follows trivially from $(i)$ and $(ii)$.
\end{proof}
Now we turn our attention to the question of character inner amenability of the Banach algebras $A\oplus_\infty B$ and  $A\oplus_p B.$ Recall that these are equipped with the usual direct product multiplications and the norms $\|(a,b)\|=max\left\{\|a\|,\|b\| \right \}$ and
  $\|(a,b)\|=(\|a\|^p+\|b\|^p)^\frac{1}{p},$  respectively.
  A direct verification shows that
 $$\bigtriangleup (A\oplus_p B)=(\bigtriangleup(A)\times \left
\{0\right\})\cup(\left\{0 \right\}\times\bigtriangleup(B)),  \ 1\leq p
\leq \infty;$$ from which we get the next result.
\begin{proposition}
Let $A$ and $B$ be Banach algebras and let $1\leq p\leq\infty$. Then
$A\oplus_p B$ is character inner amenable if and only if both $A$
and $B$ are character inner amenable.
\end{proposition}

\section{ module extension  and triangular Banach algebras}

For a Banach algebra $A$   and  a Banach $A$-module  $X$ let   $A\oplus
X$  be   the module extension Banach algebra which is equipped with the
  algebra product    $(a,x)\cdot (b,y)=(ab,ay+xb),   ( a,b\in A,     x,y\in X)$ and the norm $\|(a, x)\|=\|a\|+\|x\|.$
  The second dual $(A\oplus X)^{**}$ can be identified with $A^{**}\oplus_1
X^{**}$ as a Banach space, and it is not difficult to verify that the first Arens product on $(A\oplus X)^{**}$ is given by $(m,\lambda)\square(n,\mu)=(m\square   n,m\mu+\lambda n)$. Some aspects of the module extension  Banach algebras have been discussed in \cite{Z}.

   Let $A$ and $B$ be Banach algebras and let $X$ be a Banach $A,B$-module; that is, a left $A-$module and a right $B-$module satisfying     \ $\|axb\|\leq \|a\|\|x\|\|b\|$,  $(a\in A, \ b\in B$ \ $x\in X)$. The corresponding triangular Banach algebra $$\tau =\left\{\left(
                                             \begin{array}{cc}
                                                                 a & x \\
                                                                 0 & b \\
                                                               \end{array}
                                                             \right):  \ a\in A,  x\in X,  b\in B\right\}.$$\\
 is equipped with the usual $2\times2-$matrix operations and the norm $\|\left(
         \begin{array}{cc}
           a & x \\
           0 & b \\
         \end{array}
       \right)\|=\|a\|+\|x\|+\|b\|.$ The Arens products on the second dual of $\tau$ is studied in \cite{FM}.
Recall that the class of module extension Banach algebras includes the triangular Banach algebras. Indeed, $\tau$ can be identified with the module extension  $(A\oplus_1B)\oplus X;$ in which $X$ is considered as a $A\oplus_1B$-module under the operations $(a,b)\cdot x=ax$ and  $x\cdot(a,b)=xb.$
\begin{proposition}
Let $A$ be a Banach algebra and $X$ be a Banach $A$-module. Then for
the module extension Banach algebra $A\oplus X, \ \triangle(A\oplus X)=\bigtriangleup(A)\times\{0\}$.  In particular,  for the
triangular Banach algebra $\tau,$
$\bigtriangleup(\tau)=\triangle(A\oplus_1 B)\times\{0\}.$
\end{proposition}
\begin{proof}
Trivially $\bigtriangleup(A)\times\{0\}\subseteq\bigtriangleup
(A \oplus X).$ Let   $(\varphi,\psi)\in
\bigtriangleup (A\oplus X)$. So for $a,b\in
A$, $(\varphi,\psi)((a,0)(b,0))=(\varphi,\psi)((a,0))(\varphi,\psi)((b,0))$. It follows that
$\varphi(ab)=\varphi(a)\varphi(b)$. Also for $x,y\in X$,
$0=(\varphi,\psi)((0,x)(0,y))=(\varphi,\psi)((0,x))(\varphi,\psi)((0,y))=\psi(x)\psi(y)$. So $\psi=0$
and finally $\varphi\in \bigtriangleup (A)$. Hence $\bigtriangleup
(A\oplus X)=\bigtriangleup(A)\times\{0\}$. The second part is clear.
\end{proof}
The next result which studies  the character amenability of $A\oplus X$ and $\tau$ is a direct application of {\cite[Theorem 1.4]{KLP}} for  the module extension  $A\oplus X.$
\begin{proposition}\label{C}
Let $A$ be a Banach algebra,  $X$ be  a Banach $A$-module  and let $\varphi\in
\bigtriangleup(A)$. Then $A\oplus X$ is $(\varphi,0)$-amenable if
and only if there exists a bounded net $(a_\alpha,x_\alpha)$
   \ in $A\oplus X$ satisfying

(i)$\|aa_\alpha -\varphi(a)a_\alpha\|\rightarrow 0$ for all $a\in A$ and
$\varphi(a_\alpha)=1$ for all $\alpha$,

(ii)$\|ax_\alpha -\varphi(a)x_\alpha\|\rightarrow 0$ for all $a\in A$ and

(iii) $\|xa_\alpha\| \rightarrow 0$, for all $x\in X$.
\end{proposition}
\begin{corollary}
(i) If $A\oplus X$ is character amenable then  so is  $A$.
 The converse also holds in the case where  $XA=0$.

(ii) If $\tau$ is character amenable then both $A$ and $B$ are character amenable. The converse also holds in the case where  $XB=0$.
\end{corollary}
Similar to Proposition \ref{C} we have the next result, which is based on Proposition \ref{K},   characterizing the character inner amenability of $A\oplus X.$
\begin{proposition}
Let $A$ be a Banach algebra,  $X$ be a Banach $A$-module  and let $\varphi\in
\bigtriangleup(A)$. Then $A\oplus X$ is $(\varphi,0)$-inner amenable
if and only if there exists a bounded net
$(a_\alpha,x_\alpha)$ in $A\oplus X$ satisfying

(i) $\|aa_\alpha-a_\alpha a\|\rightarrow 0$  for all $a\in A$ and
$\varphi(a_\alpha)=1$ for all $\alpha$,

(ii) $\|xa_\alpha-a_\alpha x\|\rightarrow 0$ for all $x\in X$,   and

(iii) $\|ax_\alpha-x_\alpha a\|\rightarrow 0$ for all $a\in A$.
\end{proposition}
\begin{corollary}
 If $A\oplus X$ is character inner amenable then $A$ is character inner amenable.
 In particular if $\tau$ is character inner amenable then both $A$ and $B$ are character inner amenable.
\end{corollary}

\end{document}